\documentclass[letterpaper, 10 pt, conference]{ieeeconference}

\IEEEoverridecommandlockouts
\usepackage{amsmath,amssymb,amsfonts,theorem}
\usepackage{graphicx,epsfig,color}
\usepackage{subfigure,psfrag}
\usepackage{url}
\usepackage{bm}
\usepackage{acronym}
\usepackage[color=white!90!black, textwidth=3cm]{todonotes}
\usepackage[switch, pagewise, mathlines]{lineno}
\usepackage{algorithm}												%
\usepackage{algpseudocode}											
\usepackage{pgfplots}

\newcommand{\until}[1]{\{1,\dots, #1\}}

\renewcommand{\theenumi}{(\roman{enumi})}

\def\1{\mathbf{1}}

\renewcommand{\baselinestretch}{1}

\newcommand\oprocendsymbol{\hbox{$\square$}}
\newcommand\oprocend{\relax\ifmmode\else\unskip\hfill\fi\oprocendsymbol}

\def\V{\mathcal{V}}
\def\E{\mathcal{E}}
\def\C{\mathcal{C}}

\def\N{\mathbb{N}}
\def\R{\mathbb{R}}

\newcommand{\real}{\ensuremath{\mathbb{R}}}

\newtheorem{theorem}{Theorem}[section]

\newtheorem{assumption}[theorem]{Assumption}

{\theorembodyfont{\rmfamily} 
\newtheorem{remark}[theorem]{Remark}

}

\newcommand	{\Section}				[0]	{Section}
\newcommand	{\Sections}				[0]	{Sections}

\newcommand	{\Figure}				[0]	{Figure}

\newcommand	{\Theorem}				[0]	{Theorem}

\newcommand	{\Algorithm}			[0]	{Algorithm}

\newcommand	{\Assumption}			[0]	{Assumption}
\newcommand	{\Assumptions}			[0]	{Assumptions}

\newcommand{\DefinedAs}				[0]	{\mathrel{\mathop:}=}
\newcommand{\IDefinedAs}			[0]	{=\mathrel{\mathop:}}

\newcommand{\cmax}					[1]	{\left[ #1 \right]_{c}}

\acrodef{ADMM}		[ADMM]			{Alternating Direction Method of Multipliers}
\acrodef{MC}		[MC]			{Monte Carlo}
\acrodef{NR}		[NR]			{Newton-Raphson}
\acrodef{NRC}		[NRC]			{Newton-Raphson Consensus}
\acrodef{aNRC}		[a-NRC]			{asynchronous Newton-Raphson Consensus}
\acrodef{raNRC}		[ra-NRC]		{robust asynchronous Newton-Raphson Consensus}
\acrodef{raAC}		[ra-AC]			{robust asynchronous Average Consensus}

\title{Analysis of Newton-Raphson Consensus for multi-agent convex optimization under asynchronous and lossy communications
\thanks{This result is part of projects that have received funding from the European Union's Horizon 2020 research and innovation programme under grant agreements n.\ 636834 (DISIRE) and n.\ 638992 (ERC Starting Grant OPT4SMART), and the Swedish research council Norrbottens Forskningsr{\aa}d.}}

\author{ Ruggero Carli \and Giuseppe Notarstefano \and Luca Schenato \and Damiano Varagnolo%
  \thanks{R.\ Carli and L.\ Schenato are with the Department of Information Engineering, University of Padova, Via Gradenigo 6/a, 35131 Padova, Italy {\tt\small \{ carlirug | schenato \}@dei.unipd.it}.}
  \thanks{Giuseppe Notarstefano is with the Department of Engineering, Universit\`a del Salento, Via per Monteroni, 73100 Lecce, Italy {\tt\small giuseppe.notarstefano@unisalento.it}.}
  \thanks{D.\ Varagnolo is with the Department of Computer Science, Electrical and Space Engineering, Lule{\aa} University of Technology, Forskargatan 1, 97187 Lule{\aa}, Sweden {\tt\small damiano.varagnolo@ltu.se}.}
  }

\begin{document}

\maketitle

\begin{abstract}
  We extend a multi-agent convex-optimization algorithm named Newton-Raphson consensus to a network scenario that involves directed, asynchronous and lossy communications. We theoretically analyze the stability and performance of the algorithm and, in particular, provide sufficient conditions that guarantee local exponential convergence of the node-states to the global centralized minimizer even in presence of packet losses. Finally, we complement the theoretical analysis with numerical simulations that compare the performance of the Newton-Raphson consensus against asynchronous implementations of distributed subgradient methods on real datasets extracted from open-source databases.
\end{abstract}

\acresetall 

\section{Introduction}
\label{sec:introduction}

Distributed optimization algorithms are important building blocks in several estimation and control problems arising in peer-to-peer networks. To cope with real-world requirements, these algorithms need to be designed to work under asynchronous, directed, faulty and time-varying communications. Unfortunately, despite being the literature on distributed optimization already rich, most of the existing contributions have been proved to work in networks whose communication schemes follow synchronous, undirected, and often time-invariant information exchange mechanisms.

Early references on distributed optimization algorithms involve primal subgradient iterations~\cite{Nedic2010constrained}. Sub-gradient based algorithms have the advantage of being simple to implement and suitable for non-differentiable cost functions. Moreover, they recently have been extended to directed and time-varying communication~\cite{lin2012distributed,Nedic2013distributed}. However, these algorithms exhibit sub-linear convergence rates.

More recently, primal subgradient strategies have been proposed with guaranteed convergence in directed communication graphs~\cite{gharesifard2014} and in time-varying event-triggered communication schemes~\cite{kia2014}. However, these schemes require weight-balanced graphs, an assumption that is difficult to be satisfied in the presence of lossy communication.

A second set of contributions is based on dual decomposition schemes. The related literature is very large and we refer to~\cite{yang2011distributed} for a comprehensive tutorial on network optimization via dual decomposition. A very popular dual distributed optimization algorithm that have improved robustness in the computation and convergence rate in the case of non-strictly convex functions is the so called \ac{ADMM}. A first distributed \ac{ADMM} implementation was initially proposed in~\cite{Schizas2008consensus}, and since then several works have appeared as accounted by the survey~\cite{boyd2011distributed}. Recently, contributions have been dedicated to increase the convergence speed of this technique by means of accelerated consensus schemes~\cite{ghadimi2014optimalparameter,teixeira2013optimal}. All these algorithms have been proved to converge to the global optimum under the assumption of fixed and undirected topologies.

Another class of distributed optimization algorithms exploits the exchange of active constraints among the network nodes. A constraints consensus algorithm has been proposed in~\cite{GN-FB:07z} to solve linear, convex and general abstract programs.  These were the first distributed optimization algorithms working under asynchronous and direct communication. Recently the constraint exchange idea has been combined with dual decomposition and cutting-plane methods to solve distributed robust convex optimization problems via polyhedral approximations \cite{burger2014polyhedral}. Although well-suited for asynchronous and directed communications, these algorithms mainly solve constrained optimization problems in which the number of constraints is much smaller than the number of decision variables (or vice-versa).  

Other optimization methods include algorithms that try to exploit second-order derivatives, i.e., the Hessians of the cost functions as in~\cite{wei2013newton1,wei2013newton2}, where the distributed optimization is applied to general time-varying directed graphs. Another approach, based on Newton-Raphson directions combined with consensus algorithms, has been proposed in~\cite{zanella2011newton}: this technique works under synchronous communication, and has recently been extended to asynchronous symmetric gossip frameworks~\cite{zanella2012asynchronous}.

Importantly, all the works mentioned above require reliable communication; and, to the best of our knowledge, there is no distributed optimization algorithm that has been proved to be guaranteed to converge in the presence of lossy communication. Aiming at filling this gap, we here extend the aforementioned Newton-Raphson consensus approach in~\cite{zanella2011newton,zanella2012asynchronous} to an asynchronous, directed and unreliable network set-up. Specifically, we design a distributed algorithm which works under an asynchronous broadcast protocol over a directed graph and that is robust with respect to packet losses.

The first main contributions of this paper is to endow the Newton-Raphson algorithm in~\cite{zanella2011newton} with two additional strategies: first, a push-sum consensus method, proposed in~\cite{benezit2010weighted} to achieve average consensus in directed networks; second, a robust consensus method, proposed in~\cite{Dominguez:11} to achieve average consensus in presence of packet losses through keeping memory of the total mass of the internal states of the algorithm, so that nodes can recognize if they missed some information at a certain point, and reconstruct it.

The rationale under the combination of the push-sum and robustification protocols with the Newton-Raphson consensus is the following. In the Newton-Raphson consensus, nodes continuously update estimates of a Newton descent direction by means of an average consensus, that forces the nodes to share a common descent direction. Thus, if this averaging property is maintained under asynchronous, directed and lossy communication, the convergence properties of the descent updates can be preserved.

The second main contribution of this paper is to show that, under suitable assumptions on the initial conditions and on the step-size parameter, the Newton-Raphson consensus is locally exponentially stable around the global optimum as soon as the local costs are $\mathcal{C}^{2}$ and strongly convex with second derivative bounded from below. The exponential convergence is achieved even in the presence of lossy and broadcast communication, as long as the communication graph is strongly connected and the number of consecutive packet losses is bounded. The proof relies on a time-scale separation of the Newton descent dynamics and the average consensus one. This result thus extends the findings of~\cite{ECC15}, where the convergence was proved for the quadratic local costs case.

The third main contribution of this paper is to complement the theoretical results with numerical simulations based on real datasets extracted from an open-source database. Findings then confirm the local exponential stability and the exponential rate of convergence on a problem where the local cost functions are smooth and convex.

The paper is organized as follows: \Section~\ref{sec:problem_formulation} formulates our problem and working assumptions. \Section~\ref{sec:aNRC} then introduces the proposed algorithm and gives some intuitions on the convergence properties of the scheme, which are then summarized in \Section~\ref{sec:theoretical_analysis}. Finally, \Section~\ref{sec:numerical_experiments} collects some numerical experiments corroborating the theoretical results.

\section{Problem Formulation and Assumptions}
\label{sec:problem_formulation}

\paragraph*{Problem formulation} we consider the separable optimization problem
\begin{equation}
	\label{eq:problem}
	x^{\ast} \DefinedAs \min_x \sum_{i=1}^N f_i(x)
\end{equation}
under the assumptions that each $f_i$ is known only to node $i$ and is $\C^{2}$, and strongly convex with second derivative bounded from below, i.e., $f_{i}''(x) > c$ for all $x$ (so that $f_{i}$ is coercive). For notational convenience and w.l.o.g.\ we deal with the scalar case, i.e., $x \in \real$.

We then aim at designing an algorithm solving~\eqref{eq:problem} with the following features: 
\begin{enumerate}
	\item \emph{being distributed}: each node has limited computational and memory resources and it is allowed to communicate directly only with its in- and out-neighbors; 
	\item \emph{being asynchronous}: nodes do not share a common reference time, but rather perform actions according to local clocks independent of each other;
	\item \emph{being robust w.r.t.\ packet losses}: packets broadcast by a node may sometimes be not received by its out-neighbors due to, e.g., collisions or fading effects.
\end{enumerate}

\paragraph*{Assumptions} formally, we consider a network representable through a given, fixed, directed and strongly connected graph $\mathcal{G} = (V,\E)$ with nodes $V=\left\{1,\ldots,N\right\}$ and edges $\E\subseteq \V \times \V$ so that $(i,j) \in \E$ iff node $j$ can directly receive information from node $i$. With $\mathcal{N}_i^{\text{out}}$ we denote the set of \emph{out-neighbors} of node $i$, i.e., $\mathcal{N}_i^{\text{out}} \DefinedAs \left\{j \in \V \; | \; (i,j) \in \E \right\}$ is the set of nodes receiving messages from $i$. Similarly, with $\mathcal{N}_i^{\text{in}}$ we denote the set of \emph{in-neighbors} of $i$, i.e., $\mathcal{N}_i^{\text{in}} \DefinedAs \left\{j \in \V \; | \; (j,i) \in \E\right\}$.

As for the concept of time, we assume that each node has its own clock that locally and independently triggers when to transmit. With $\sigma(t) \in \until{N}$, $t = 1, 2, \ldots$ be the sequence identifying the generic triggered node at time $t$, i.e., $\sigma(1)$ is the first triggered node, $\sigma(2)$ the second, etc., so that $\sigma(t)$ is a process on the alphabet $\until{N}$. When a node is triggered, it performs some local computation and then broadcasts some information to its out-neighbors. Due to unreliable communication links, this information can be potentially lost. 

We assume that to solve~\eqref{eq:problem} each node $i$ stores in its memory a local copy, say $x_i$ (also called \emph{local estimate} or \emph{local decision variable}), of the global decision variable $x$. With this new notation~\eqref{eq:problem} reads as
\begin{equation}
	\min_{x_1,\ldots, x_N} \; \sum_{i=1}^N f_i(x_i) \quad \text{s.t.} \; x_i = x_j \; \text{for all} \; (i,j) \in \E.
\label{eq:problem_rif}
\end{equation}
Notice that the strong connectivity of graph $\mathcal{G}$ ensures then that the optimal solution of~\eqref{eq:problem_rif} is given by $x_1=\ldots=x_N=x^*$, i.e., ensures that problems~\eqref{eq:problem} and~\eqref{eq:problem_rif} are equivalent.

\section{The \acl{raNRC} Algorithm}
\label{sec:aNRC}

We now introduce an algorithm suitable for solving problem~\eqref{eq:problem} under the asynchronous and lossy communication assumptions posed in \Section~\ref{sec:problem_formulation}. The procedure, called \ac{raNRC} and reported in \Algorithm~\ref{alg:raNRC}, has been initially presented in~\cite{ECC15} but is reported here for completeness and ease of reference. In the pseudo-code we assume w.l.o.g.\ $\sigma(t) = i$, i.e., that the node that triggers at iteration $t$ is the node $i$.

We assume that every node $i$ stores in its memory the variables $x_i$, $g_{i}$, $h_{i}$, $y_i$, $z_i$, $b_{i,y}$, $b_{i,z}$, and $r_{i,y}^{(j)}$, $r_{i,z}^{(j)}$ for every $j \in \mathcal{N}^{\text{in}}_i$, with the following meanings:
\begin{itemize}
	\item $x_i$ represents the current local estimate at node $i$ of the global minimizer $x^*$;
	\item $g_i$ and $h_i$ represent some specific function of the first and second derivatives of the local cost $f_i(x_i)$ computed at the current value of $x_i$. $g_i^{\text{old}}$ and $h_i^{\text{old}}$ represent the old values of $g_{i}$ and $h_{i}$ at the previous local step; 
	\item $y_i$ and $z_i$ represent respectively the local estimate at node $i$ of the global sums $\sum_i g_i$ and $\sum_i h_i$;
	\item $b_{i,y}$ and $b_{i,z}$ represent respectively quantities that are used by node $i$ to locally keep track of the total mass of the internal states $y_i$ and $z_i$. Notice that $b_{i,y}$ and $b_{i,z}$ are the only local variables that are broadcast by node $i$ to its out-neighbors;
	\item $r_{j,y}^{(i)}$ and $r_{j,z}^{(i)}$ represent respectively quantities that are used by node $j$ to locally keep track of the total mass of the internal states $y_i$ and $z_i$ of $i$, that are in general inaccessible by $j$. In other words, with $r_{j,y}^{(i)}$ and $r_{j,z}^{(i)}$ node $j$ tracks the status of node $i$: when the communication link from $i$ to $j$ does not fail, then node $j$ updates $r_{j,y}^{(i)}$ and $r_{j,z}^{(i)}$ with the received $b_{i,y}$ and $b_{i,z}$. Otherwise, when the communication link from $i$ to $j$ fails, then $r_{j,y}^{(i)}$ and $r_{j,z}^{(i)}$ remain equal to the previous total mass received.
\end{itemize} 
Thus the \ac{raNRC} algorithm builds on top of broadcast-like average consensus protocols~\cite{benezit2010weighted} (i.e., the structure of the updates of the variables $y_{i}$ and $z_{i}$) and of strategies for handling packets losses in consensus schemes~\cite{Dominguez:11} (i.e., the way of using the variables $b_{i,y}$, $b_{i,z}$, $r_{j,y}^{(i)}$ and $r_{j,z}^{(i)}$ to prevent information losses through mass-tracking robust strategies).

We also notice that the algorithm exploits the thresholding operator
\begin{equation*}
	\cmax{ z }
	\DefinedAs
	\left\lbrace
	\begin{array}{ll}
		z & \text{ if } z \geq c \\
		c & \text{ otherwise.} 
	\end{array}
	\right.
\end{equation*}
where $c$ is a positive scalar to be properly chosen to avoid division-by-zero in the algorithm. 

\begin{algorithm}
\caption{\acf{raNRC}}
    \begin{algorithmic}[1]
		\State on wake-up, and before transmission, node $i$ updates its local variables as
		\begin{align*}
		y_i & \leftarrow \frac{1}{|\mathcal{N}^{\text{out}}_i|+1}\left[y_i + g_i-g_i^{\text{old}}\right]\\
		z_i & \leftarrow \frac{1}{|\mathcal{N}^{\text{out}}_i|+1}\left[z_i + h_i-h_i^{\text{old}}\right]\\
		g_i^{\text{old}} &\leftarrow g_i\\
		h_i^{\text{old}} &\leftarrow h_i\\
		x_i & \leftarrow (1-\varepsilon) x_i + \varepsilon \frac{y_i}{\cmax{z_i}}\\
		g_i& \leftarrow f''_i(x_i)x_i-f'_i(x_i)\\
		h_i &\leftarrow f''_i(x_i)\\
		b_{i,y} & \leftarrow b_{i,y}+y_i\\
		b_{i,z} & \leftarrow b_{i,z}+z_i
		\end{align*}
		\State node $i$ then broadcasts $b_{i,y}$ and $b_{i,z}$ to its neighbors;
		\State every out-neighbor $j \in \mathcal{N}_i^{\text{out}}$ updates (if receiving the packet, otherwise it does nothing) its local variables as
		\begin{align*}
		y_j & \leftarrow b_{i,y}-r_{j,y}^{(i)}+ y_j+ g_j-g_j^{\text{old}} \\
		z_j & \leftarrow b_{i,z}-r_{j,z}^{(i)}+  z_j + h_j-h_j^{\text{old}}\\
		g_j^{\text{old}} &\leftarrow g_j\\
		h_j^{\text{old}} &\leftarrow h_j\\
		x_j & \leftarrow (1-\varepsilon) x_j + \varepsilon \frac{y_j}{\cmax{z_j}}\\
		g_j& \leftarrow f''_i(x_j)x_i-f'_i(x_j)\\
		h_j &\leftarrow f''_i(x_j)\\
		r_{j,y}^{(i)}& \leftarrow b_{i,y}\\
		r_{j,z}^{(i)}& \leftarrow b_{i,z}
		\end{align*}
    \end{algorithmic}
\label{alg:raNRC}
\end{algorithm}

\paragraph*{Initialization of the \ac{raNRC} algorithm} we assume that every agents perform the following initialization step of the local variables: let $x^o$ be a common initial estimate of the global minimizer (may be chosen equal to zero for convenience). Then
\begin{align*}
	&x_i = x^o\\
	&y_i = g_i^{\text{old}}=g_i=f''_i(x^o)x^o-f'_i(x^o) \IDefinedAs y_{i}^{o} \\
	&z_i = h_i^{\text{old}}=h_i=f''_i(x^o) \IDefinedAs z_{i}^{o} . \\
\end{align*}

\section{Informal description of the convergence properties of the algorithm}
\label{sec:informal_description_of_the_convergence_properties_of_the_algorithm}

We now provide an intuitive verbal description of the main features and intuitions behind the proposed algorithm, before presenting a mathematical characterization in the following \Section~\ref{sec:theoretical_analysis}.

We start by noticing that the only free parameter of the algorithm is given by the scalar $\varepsilon \in (0,1]$. This parameter is fundamental since it regulates the trade-off between the stability of the algorithm and its speed of convergence. Indeed the algorithm is characterizable through two distinct dynamics: a fast one, which distributedly computes averages of the $y_{i}$'s and $z_{i}$'s based on a robust consensus algorithm, and a slow dynamics, that estimates the minimizer of the global cost function using the ratio of the averaged $y_{i}$'s and $z_{i}$'s as a Newton direction. More specifically, the variables $x_i$ are associated to the slow dynamics, while all the other variables $y_i,z_i,g_i,h_i,b_{i,z},b_{i,y},r^{(i)}_{i,y},r^{(i)}_{i,z}$ are associated to the fast dynamics.

The parameter $\varepsilon$ regulates then the separation of these two time scales: the smaller $\varepsilon$ is, the larger this separation is, so that small $\varepsilon$'s imply slow distributed averaging of the $y_{i}$'s and $z_{i}$'s. On the other hand, the rate of convergence of the slow dynamics, i.e., of the Newton-Raphson on the $x_{i}$'s, can be shown to be locally given by $(1-\varepsilon)$; therefore small $\varepsilon$'s imply also slower convergence towards the global optimum. 

In the following we use the symbol $\to$ to indicate the behavior of a certain variable as the number of iterations of \Algorithm~\ref{alg:raNRC} goes to infinity, while we reserve $\leftarrow$ for denoting values assignment operations (e.g., $x_{i} \leftarrow x^{o}$ reads as ``variable $x_{i}$ assumes the value $x^{o}$'').

\subsection{Intuitions behind the fast dynamics: the case $\varepsilon = 0$}

As $\varepsilon$ approaches zeros, $x_i$ changes very little from one iteration to the other, i.e., $x_i\approx \mathrm{cost.}$. Indeed if we assume $\varepsilon = 0$, then the local estimate update rule becomes $x_i \leftarrow x_i$, so that $g_i \leftarrow g_i^{\text{old}}$ and $h_i \leftarrow h_i^{\text{old}}$, i.e., constant values. Therefore in this case the dynamics of $y_i$ only depends on its initial value $f''_i(x_i)x_i - f_i'(x_i)$ and on the communication sequence. Similar considerations hold for $z_i$'s. Thus in this case the variables $y_i$ and $z_i$ evolve as the robust ratio consensus described in~\cite{Dominguez:11}, i.e.,
$$
	y_i \to \rho_i \left( \frac{1}{N} \sum_{i = 1}^{N} \Big( f''_i(x_i)x_i- f_i'(x_i) \Big) \right)
$$
$$
z_i \to \rho_i  \left( \frac{1}{N} \sum_{i = 1}^{N} f''_i(x_i) \right)
$$
where $0 < \rho_i \leq 1$ is some scalar that depends on the packet loss sequence. Thus, regardless of the specific communications and packet losses sequence,
$$
	\frac{y_i}{z_i} \to \frac{\sum_i f''_i(x_i)x_i-f_i'(x_i)}{\sum_i f''_i(x_i)}=:\phi(x_1,\ldots,x_N)
$$
i.e., all the local ratios $\displaystyle \frac{y_i}{z_i}$ converge to the same value $\phi$. 

\subsection{Intuitions behind the slow dynamics: the case $\displaystyle \frac{y_i}{z_i} = \phi(x_1,\ldots,x_N)$}

The slow dynamics can be obtained by assuming that the fast dynamics has converged to steady-state value considering $\varepsilon=0$. The idea is that if $\varepsilon\approx 0$, then also $\displaystyle \frac{y_i(k)}{z_i(k)}\approx\phi(x_1,\ldots,x_N)$. In this scenario, the dynamics of each local variable $x_i$ can then be written as
$$
	x_i \leftarrow (1-\varepsilon) x_i + \varepsilon \phi(x_1,\ldots,x_N), \quad i = 1,\ldots,N .
$$
This implies that all the various agents update the local values with the same identical rule; thus nodes behave in this case as $N$ identical systems that are driven by the same forcing term. This implies that any difference in the initial value of  $x_i$ will vanish, eventually leading to
$$
	x_i \to  x, \quad \forall i=1,\ldots, N.
$$
In this case, moreover,
$$
\phi(x_1,\ldots,x_N) \to \frac{\sum_i f''_i(x)x-f_i'(x)}{\sum_i f''_i(x)}=x-\frac{\overline{f}' (x)}{\overline{f}''(x)}
$$
where $\overline{f}(x) \DefinedAs \sum_i f_i(x)$. Thus the dynamics of the local variables are of the form
$$
	x
	\leftarrow
	(1-\varepsilon)x + \varepsilon \left(x-\frac{\overline{f}' (x)}{\overline{f}''(x)}\right)
	=
	x - \varepsilon \frac{\overline{f}' (x)}{\overline{f}''(x)} ,
$$
i.e., a Newton-Raphson algorithm that, under the posed smoothness assumptions on the local $f_{i}$'s, converges to the solution of~\eqref{eq:problem}. Thus,
$$
	x_i \to x^* \qquad \forall i=1,\ldots, N, \quad \forall x^{o} \in \real.
$$

\subsection{Intuitions behind the local rate of convergence $1-\varepsilon$}

The previous analysis allows to estimate the rate of convergence around the global minimum $x^*$. In fact, if we assume a sufficiently large separation of time scales (i.e., the average consensus on the $y_{i}$'s and $z_{i}$'s to be much faster than the Newton-Raphson dynamics), then the rate of convergence of the whole algorithm is dominated by the slow dynamics. If then one further assumes the $f_i(x)$ to be $\mathcal{C}^{3}$ then the Newton-Raphson dynamics can be linearized so to obtain
$$
	\left. \frac{d}{dx}  \frac{\overline{f}' (x)}{\overline{f}''(x)} \right|_{x=x^*}
	=
	\left.
		\frac{\overline{f}''(x)}{\overline{f}''(x)}
		-
		\frac{\overline{f}' (x)\overline{f}'''(x)}{(\overline{f}''(x))^2} \right|_{x=x^*}
	=
	1
$$
where we used the fact that $\overline{f}'(x^*)=0$. Therefore, the dynamics of the Newton-Raphson component of the algorithm around the equilibrium point $x^*$ can be written as
$$
	x^+ \approx x - \varepsilon (x-x^*) \Rightarrow (x - x^*)^+ \approx (1-\varepsilon)(x - x^*),
$$
which clearly shows that locally the rate of convergence is exponential with a rate given by $(1-\varepsilon)$. This confirms the previous intuition that smaller $\varepsilon$'s lead to slower convergence rates.

\subsection{Intuitions behind the stability properties of the \ac{raNRC} algorithm}

As discussed above, $\varepsilon$ dictates the relative speed of the fast dynamics (driving the variables $y_{i}$ and $z_{i}$ to a consensus), and the slow dynamics for the Newton-Raphson-like evolution of the local estimates $x_{i}$.  The parameter $\varepsilon$, moreover, dictates how much each node $i$ trusts $\displaystyle \frac{y_{i}}{z_{i}}$ as a valid Newton direction. During the transient, indeed, this ratio is not the Newton direction of neither the local nor the global cost computed at the current $x_{i}$.

Clearly, if the consensus on the $y_{i}$'s and $z_{i}$'s is much faster than the evolution of the $x_{i}$'s (i.e., if $\varepsilon$ is ``small enough'') then one can expect that the aforementioned separation of time scales holds, so that all the quantities converge to their equilibria and the overall algorithm converges. But if $\varepsilon$ is not sufficiently small then the stability of the overall system is not guaranteed: indeed, in the following section we prove that there always exists a suitable critical value $\varepsilon_c$ such that for all  $0<\varepsilon<\varepsilon_c$ the algorithm is locally exponential stable, while nothing can be said for $\varepsilon > \varepsilon_{c}$.

Notice that estimating (even offline) such $\varepsilon_c$ is a very difficult task, and that explicit bounds are often very conservative. Unfortunately, moreover, the difficulty of finding conservative bounds on $\varepsilon_{c}$ conflicts with the practical necessity of having high $\varepsilon$'s (the higher $\varepsilon$, the faster the algorithm converges -- if converging -- to the optimum).

\section{Theoretical analysis of the \acl{raNRC}}
\label{sec:theoretical_analysis}

We now provide a theoretical analysis of the proposed algorithm under asynchronous and lossy communication scenarios. In particular we provide some sufficient conditions that guarantee local exponential stability under the assumptions posed in \Section~\ref{sec:problem_formulation}. We thus extend our previous work~\cite{ECC15}, dedicated to the quadratic local costs case, to more generic local convex costs.

Informally, we assume that each node updates its local variables and communicates with its neighbors infinitely often, and that the number of consecutive packet losses is bounded. Formally, we assume that:
\begin{assumption}[Communications are persistent]
	For any iteration $t \in \N $ there exists a positive integer number $\tau$ such that each node performs at least one broadcast transmission within the interval $[t, t+\tau]$, i.e., for each $i \in \until{N}$ there exists $t_i \in [t, t+\tau]$ such that $\sigma(t_i)=i$.
	\label{ass:persistent_communications}
\end{assumption}
\begin{assumption}[Packet losses are bounded]
	There exists a positive integer $L$ such that the number of consecutive communication failures over every directed edge in the communication graph is smaller than $L$.
	\label{ass:bounded_packet_losses}
\end{assumption}

Notice that \Assumption~\ref{ass:bounded_packet_losses} could be relaxed, in the sense that we may allow links to fail permanently as soon as the remaining network still exhibits strong connectivity. The following result summarizes our characterization of the convergence properties of the \ac{raNRC} algorithm:
\begin{theorem}
	Under \Assumptions~\ref{ass:persistent_communications},~\ref{ass:bounded_packet_losses} and the assumptions posed in \Section~\ref{sec:problem_formulation} there exist some positive scalars $\varepsilon_c$ and $\delta$ s.t.\ if the initial conditions $x^o \in \R$ satisfy $|x^o-x^*|<\delta$ and if $\varepsilon$ satisfies $0 < \varepsilon < \varepsilon_c$ then the local variables $x_i$ in \Algorithm~\ref{alg:raNRC} are exponentially stable w.r.t.\ the global minimizer $x^*$.
	\label{thm:convergence}
\end{theorem}
\begin{proof} The proof of this theorem is quite involved and relies on many intermediate results. In the interest of space we refer the interested reader to a longer version of this work, \cite{CDC15}, including all the technical details in a dedicated Appendix. 
\end{proof}

Introducing the notation $x_{i}(t)$ to indicate the value $x_{i}$ after the $t$-th broadcast event in the whole network, \Theorem~\ref{thm:convergence} reads as follows: if the hypotheses are satisfied then there exist positive scalars $C$ and $\lambda<1$, possibly function of $\delta$ and $\varepsilon$, s.t.
$$
	|x_i(t)-x^*|\leq C \lambda^t ,
	\qquad
	t = 1, 2, \ldots
$$

\begin{remark}
	\Algorithm~\ref{alg:raNRC} assumes the initial conditions of the local variable $x_i$ to be all identical to $x^o$. Although being not a very stringent requirement, this assumption can be relaxed. I.e., slightly modified versions of \Theorem~\ref{thm:convergence} would hold even in the case $x_{i} = x_{i}^{o}$ as soon as all the initial conditions are sufficiently close to the global minimizer $x^*$, i.e., as soon as $|x_i^o-x^*|<\delta$ for all $i = 1, \ldots, N$.
\end{remark}

\begin{remark}
	The initial conditions on the local variables $y_i = g_i^{\text{old}}=g_i=f''_i(x^o)x^o-f'_i(x^o)$ and $z_i = h_i^{\text{old}}=h_i=f''_i(x^o)$ are instead more critical for the convergence of the local variables $x_i$ to the true minimizer $x^*$. As shown in~\cite{zanella2011newton}, small perturbations of these initial conditions can lead to convergence to a point $\overline{x} \neq x^*$ (notice that these perturbations do not affect the stability of the algorithm, so that possible small numerical errors due to the computation and data quantization do not disrupt the convergence properties of the algorithm). Nonetheless the map from the amplitude of these perturbations and the distance $\left\| \overline{x} - x^* \right\|$ is continuous, so that if these perturbations are small then $\overline{x} \approx x^*$.
\end{remark}

\begin{remark}
Although the previous theorem guarantees only local exponential convergence, numerical simulations on real datasets seem to indicate that the basin of attraction is rather large and stability is mostly dictated by the choice of the parameter $\varepsilon$.
\end{remark}

\section{Numerical Experiments}
\label{sec:numerical_experiments}

First, we empirically study the sensitivity of the convergence speed of the proposed \ac{raNRC} algorithm on $\varepsilon$ and on the packet loss probability in \Sections~\ref{ssec:empirical_analysis_of_the_effects_of_varepsilon} and~\ref{ssec:empirical_analysis_of_the_effects_of_packet_losses}, respectively. Then, we compare in \Section~\ref{ssec:convergence_speeds_comparisons} the convergence speed of the \ac{raNRC} against the speed of asynchronous subgradient schemes.

We consider the network depicted in \Figure~\ref{fig:network_NRC_vs_SG} and apply our algorithm in the context of robust regression using real-world data. Specifically we consider a database $\mathcal{D}$ containing financial information on various houses. To each house $j$ there is associated an output variable $y_{j} \in \real$, which indicates its monetary value, and a vector $\chi_{j} \in \real^{n}$, which represents $n$ numerical attributes of the $j$-th house (e.g., per capita crime rate by town, index of accessibility to radial highways, etc.). The database is distributed, i.e., the set $\mathcal{D}$ comes from $N$ different sellers that do not want to disclose their private information. More specifically, each seller $i$ owns a subset $\mathcal{D}_i$ of the global dataset $\mathcal{D}$ so that $\cup_i \mathcal{D}_i = \mathcal{D}$. Nonetheless sellers want to collectively build an estimator of the prices of new houses that is based on all the information possessed by the peers. An approach to solve this distributed regression problem is to solve an optimization problem where the local costs are given by the smooth Huber costs
\begin{equation}
    f_i\left( x \right)
    \DefinedAs
    \sum_{j \in \mathcal{D}_i}
    \frac
    {
        \left(
            y_{j} - \chi_{j}^T x - x_{0} 
        \right)^{2}
    }
    {
        \left|
            y_{j} - \chi_{j}^T x - x_{0} 
        \right|
        +
        \beta
    }
    +
    \gamma \left\| x \right\|^{2}_{2}
    \label{equ:housing_regression}
\end{equation}
where $\gamma$ is a global regularization parameter that is, for our purposes, considered to be known to all agents. We then consider a dataset $\mathcal{D}$ with $|\mathcal{D}| = 500$ elements from the Housing UCI repository\footnote{\texttt{http://archive.ics.uci.edu/ml/datasets/Housing}}, randomly assigned to $N=15$ different users communicating as in graph of \Figure~\ref{fig:network_NRC_vs_SG}. For each element we consider $n = 9$ features (the first 9 ones in the database), so that the corresponding optimization problem is 10-dimensional. The centralized optimum $x^{\ast}$ for this problem has been computed using a centralized \ac{NR} scheme with Newton step chosen with backtracking, and terminating when the Newton decrement was $< 10^{-9}$.

\begin{figure}[!htbp]
	\centering
	\includegraphics{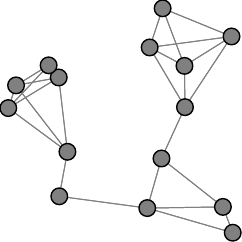}
\caption{A random geometric graph with connectivity radius 0.35.}
\label{fig:network_NRC_vs_SG}
\end{figure}

\subsection{Empirical analysis of the effects of $\varepsilon$ on the convergence speed of the \ac{raNRC} algorithm}
\label{ssec:empirical_analysis_of_the_effects_of_varepsilon}

We consider a probability of packet losses fixed to $0.1$, and a $\varepsilon$ that ranges in $\{ 10^{-4}, 10^{-3}, 10^{-2}, 10^{-1} \}$, and compare in \Figure~\ref{fig:comparison_of_average_errors_for_NR_for_different_varepsilon_and_fixed_packet_losses} the evolution of the average errors for different values of $\varepsilon$. We notice how the results agree with the intuitions developed in the previous sections: first, there may be some $\varepsilon$'s (e.g., $\varepsilon = 10^{-1}$) that lead to non converging behaviors. Second, choosing a too conservative $\varepsilon$ (e.g., $\varepsilon = 10^{-4}$ or $\varepsilon = 10^{-3}$) leads to a very slow convergence since $\displaystyle x_i \leftarrow (1-\varepsilon) x_i + \varepsilon \frac{y_i}{\cmax{z_i}}$.

\begin{figure}[!htbp]
	\centering
	\includegraphics{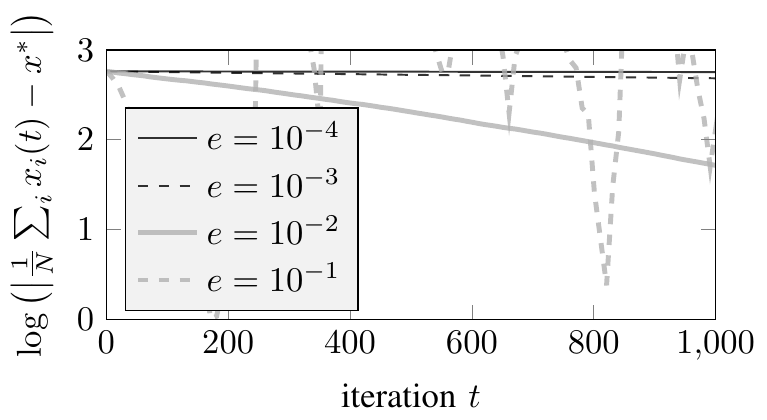}
\caption{Comparison of the evolutions of the trajectories of the average errors for different values of $\varepsilon$ and a packets loss probability $p = 0.1$.}
\label{fig:comparison_of_average_errors_for_NR_for_different_varepsilon_and_fixed_packet_losses}
\end{figure}

\subsection{Empirical analysis of the effects of packet losses on the convergence speed of the \ac{raNRC} algorithm}
\label{ssec:empirical_analysis_of_the_effects_of_packet_losses}

We consider a parameter $\varepsilon$ fixed to $0.01$, and a probability of packet losses that ranges in $\{ 0, 0.2, 0.4, 0.6 \}$, so to compare in \Figure~\ref{fig:comparison_of_average_errors_for_NR_for_different_packet_losses_and_fixed_varepsilon} the evolution of the average errors for different packets unreliability levels. We notice that, as expected, the severity of the packet losses negatively affects the convergence speed. Nonetheless the overall slowing effect is not disruptive, in the sense that even severe packet loss probabilities (namely, 0.6) do not lead to meaningless estimates.

\begin{figure}[!htbp]
	\centering
	\includegraphics{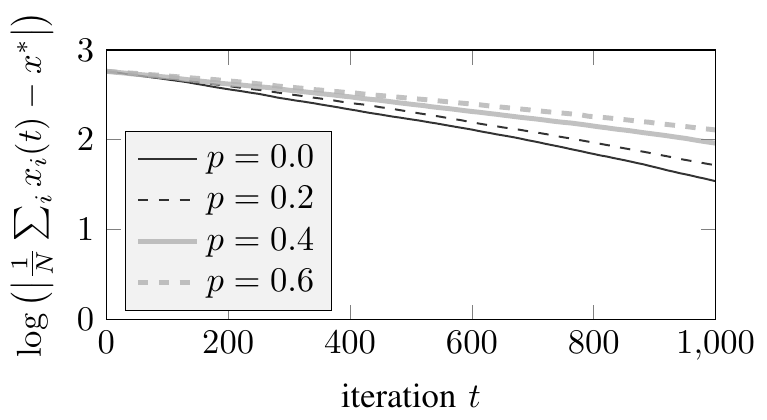}
	\caption{Comparison of the evolutions of the trajectories of the average errors for the \ac{raNRC} algorithm for different values of the packet loss probabilities and $\varepsilon = 0.01$.}
\label{fig:comparison_of_average_errors_for_NR_for_different_packet_losses_and_fixed_varepsilon}
\end{figure}

\subsection{Convergence speeds comparisons}
\label{ssec:convergence_speeds_comparisons}

We consider the asynchronous subgradient scheme reported in \Algorithm~\ref{alg:SG}, and numerically compare its convergence properties against the proposed \ac{raNRC} scheme under a packet losses probability equal to $0.1$.

\begin{algorithm}
\caption{Distributed Subgradient}
    \begin{algorithmic}[1]
		\State on initialization, each node $i$ initializes $x_{i}$ as $x_{i}^{o}$ and $t_{i}$ (the local counter of the number of updates) to $1$;
		\State on wake up, node $i$ broadcasts $x_{i}$ and $f_{i} \left( x_{i} \right)$ to all its neighbors;
		\State every out-neighbor $j \in \mathcal{N}_i^{\text{out}}$ updates (if receiving the packet, otherwise it does nothing) its local variables as
		\begin{align*}
			x_{j} & \leftarrow \frac{1}{2} \left( x_{i} + x_{j} \right) + \frac{\alpha}{t_{j}} \big( f_{i} (x_{i}) + f_{j} (x_{j}) \big) \\
			t_{j} & \leftarrow t_{j} + 1 
		\end{align*}
    \end{algorithmic}
\label{alg:SG}
\end{algorithm}

For both algorithms we compute, through gridding, that parameter ($\varepsilon$ for the \ac{raNRC}, $\alpha$ for the subgradient) that leads to the best performance in terms of convergence speed of the average guess over the various agents. We then report the evolution of the average guess over time in \Figure~\ref{fig:comparison_of_average_errors_trajectories_for_smooth_huber_costs}, and notice how the higher order information used by the \ac{raNRC} scheme over the subgradient one positively affects the asymptotic convergence speed of the procedure.

\begin{figure}[!htbp]
	\centering
	\includegraphics{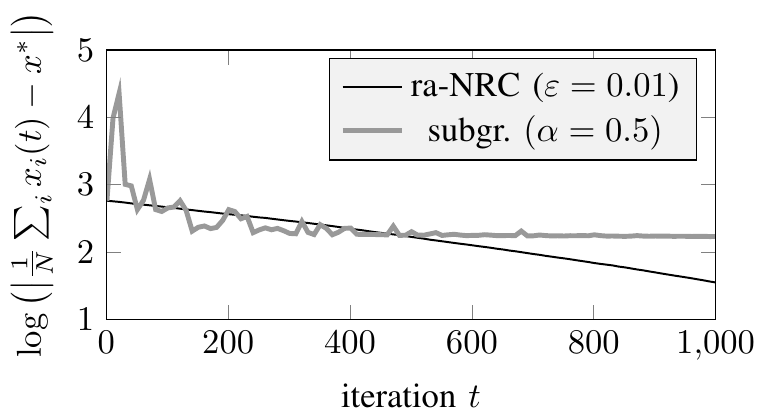}
\caption{Comparison of the evolutions of the trajectories of the average errors for the algorithms tuned with their best parameters and a packet loss probability $p = 0.1$.}
\label{fig:comparison_of_average_errors_trajectories_for_smooth_huber_costs}
\end{figure}

\section{Conclusions}
\label{sec:conclusions}

Implementations of distributed optimization methods in real-world scenarios require strategies that are both able to cope with real-world problematics (like unreliable, asynchronous and directed communications), and converge sufficiently fast so to produce usable results in meaningful times. Here we worked towards this direction, and improved an already existing distributed optimization strategy, previously shown to have fast convergence properties, so to make it tolerate the previously mentioned real-world problematics.

More specifically, we considered a robustified version of the Newton-Raphson consensus algorithm originally proposed in~\cite{zanella2011newton} and proved its convergence properties under some general mild assumptions on the local costs. From technical perspectives we shown that under suitable assumptions on the initial conditions, on the step-size parameter, on the connectivity of the communication graph and on the boundedness of the number of consecutive packet losses, the considered optimization strategy is locally exponentially stable around the global optimum as soon as the local costs are $\mathcal{C}^{2}$ and strongly convex with second derivative bounded from below.

We also shown how the strategy can be applied to real world scenarios and datasets, and be used to successfully compute optima in a distributed way.

We then notice that the results offered in this manuscript do not deplete the set of open questions and plausible extensions of the Newton Raphson consensus strategy. We indeed devise that the algorithm is potentially usable as a building block for distributed interior point methods, but that some lacking features prevent this development. Indeed it is still not clear how to tune the parameter $\varepsilon$ online so that the convergence speed is dynamically adjusted (and maximized), how to account for equality constraints of the form $Ax = b$, and how to update the local variables $x_{i}$ using partition-based approaches so that each agent keeps and updates only a subset of the components of $x$.



\end{document}